\documentclass[12pt,a4paper]{article}

\usepackage{bbm}
\usepackage{graphicx}
\usepackage[mathscr]{euscript}
\usepackage{subcaption}
\usepackage{hyperref}
\usepackage{epstopdf}
\usepackage{enumerate}
\usepackage{amsmath,amsfonts,amssymb,amsthm,epsfig,epstopdf,titling,url,array}
\usepackage{color}
\usepackage{framed}
\usepackage[utf8]{inputenc}
\usepackage[english]{babel}
\newtheorem{Remark}{Remark}[section]

 \usepackage{xparse}
 
 \NewDocumentCommand{\INTERVALINNARDS}{ m m }{
 	#1 {,} #2
 }
 \NewDocumentCommand{\interval}{ s m >{\SplitArgument{1}{,}}m m o }
 {
 	\IfBooleanTF{#1}{
 		\left#2 \INTERVALINNARDS #3 \right#4
 	}{
 		\IfValueTF{#5}{
 			#5{#2} \INTERVALINNARDS #3 #5{#4}
 		}{
 			#2 \INTERVALINNARDS #3 #4
 		}
 	}
 }
 \newtheorem{theorem}{Theorem}[section]

 \newtheorem{lemma}[theorem]{Lemma}
 \newtheorem{example}{Example}[section]
 \newtheorem{Proposition}[theorem]{Proposition}
 \newtheorem{obs}[theorem]{Observation}

\usepackage{authblk}

\begin{document} 
\title{Quadratic and cubic  Newton maps of rational functions}
  
\author[1]{Tarakanta Nayak\footnote{ tnayak@iitbbs.ac.in}}
\author[1]{Soumen Pal\footnote{Corresponding author, sp58@iitbbs.ac.in}}
\affil[1]{School of Basic Sciences,
Indian Institute of Technology Bhubaneswar, India}
 
\maketitle
\begin{abstract}
The dynamics of all quadratic Newton maps of rational functions are completely described. The Julia set of such a map is found to be either a Jordan curve or totally disconnected. It is proved that no Newton map with degree at least three of any rational function is conformally conjugate to a unicritical polynomial(i.e., with exactly one finite critical point). However, there are cubic Newton maps which are conformally conjugate to other polynomials. The Julia set of such a Newton map is shown to be a closed curve. It is a Jordan curve whenever the Newton map has two attracting fixed points.  
\end{abstract}
\section{Introduction}
The Newton map of a rational function $R$ is given by $N_R(z)=z-\frac{R(z)}{R'(z)}$. Its study is a classical topic of research for the case when $R$ is a polynomial. The iteration of the  Newton map of a polynomial leads to a useful and well-known way of finding the roots of the polynomial. There is a large body of literature on this subject (See for example~\cite{campos,hubbard,lei,von}), some of them going beyond polynomial root finding. In particular, this has  served as a motivation for  {Complex Dynamics}, the study of iteration of rational  functions in general. 
In spite of this, the Newton maps of non-polynomial rational functions seem to remain under-investigated. The Newton map of such a rational function differs from that of a polynomial in several important ways. Probably the most obvious one is that $N_R$ can have finite extraneous fixed points (these are precisely the poles of $R$). Another difference arising due to the same reason is that the degree of $N_R$ can be bigger than the number of distinct roots of $R$. Indeed, the number of distinct poles of $R$ also contributes to the degree of $N_R$ (please see Section 2 for the exact statement).
\par  Buff and Henriksen~\cite{2} studied Konig's methods of which the Newton method is a particular case. Their work is mostly focused on polynomials with several remarks on rational functions. 
\par  Barnard et al. \cite{Conj} described all the rational functions $R$ whose Newton maps are quadratic.  Two rational functions $f_1$ and $f_2$ are called conformally conjugate if there is a M\"{o}bius map $\phi$ such that $f_1 =\phi \circ f_2 \circ \phi^{-1}$. If a quadratic Newton map is conformally conjugate to a polynomial then the polynomial is $ z^2+c$ for some $c\in \mathbb{C}$. Barnard and co-authors observed that  this  $c$ belongs to $\interval[{0,\frac{1}{4}})$ and its Julia set is connected. In order to proceed with the discussion  we need few definitions. For a non-constant rational function $f$ with degree at least two, its Fatou set is defined as $\mathcal{F}(f)=\{z\in \widehat{\mathbb{C}}: \text{ $z$ has a neighborhood where} ~\{f^n\}_{n>0}~ \text{is normal}\}$. It is denoted by $\mathcal{F}(f)$. The Julia set, denoted by $\mathcal{J}(f)$, is the complement of $\mathcal{F}(f)$ in the extended complex plane $\widehat{\mathbb{C}}$. By dynamics of a rational function we mean its Fatou set and Julia set. A point $z_0$ is called a fixed point of a rational function $f$ if $f(z_0)=z_0.$ Its multiplier $\lambda$ is defined as $f'(z_0)$ if $z_0 \in \mathbb{C}$ and as $g'(0)$ when $z_0 =\infty$, where $g(z)=\frac{1}{f(\frac{1}{z})}$. A fixed point $z_0$ is called attracting or repelling if $|\lambda|<1$ or $|\lambda|>1$ respectively. An attracting fixed point is called superattracting if its multiplier is $0$. It is called indifferent  if $|\lambda|=1$. If $\lambda$ is an $n$-th root of unity then the fixed point is called rationally indifferent. A fixed point of $N_R$ is called extraneous if it is not a root of $R$. The Fatou set is an open set by definition. A component $U$ of the Fatou set of $f$ is called an invariant attracting domain if it contains an attracting fixed point of $f$. An invariant attracting domain $U$ is called completely invariant if its full pre-image $f^{-1}(U)$ is contained in $U$. Resuming the discussion of the work of Barnard et al., it is not difficult to see that the Fatou set of $z^2+c$ is  the union of two completely invariant attracting domains and the Julia set is a Jordan curve for $c\in \interval[{0,\frac{1}{4}})$.  Since the Julia set of two conformally conjugate rational functions are the same up to the application of a M\"{o}bius map (See Lemma~\ref{conjugacy} for exact statement), the Julia set of a quadratic $N_R$ which is conformally conjugate to some polynomial is a Jordan curve.
\par 
 We consider all the quadratic Newton maps including those conformally conjugate to  polynomials and describe their dynamics completely. It is seen that the Julia set is either a Jordan curve or totally disconnected (i.e., each component of the Julia set is singleton). Though the Julia sets of quadratic rational functions are known to be either connected or totally disconnected  ~\cite{Mil01}, a connected Julia set is not necessarily a Jordan curve. For example, the Julia set of $z^2-1$ is connected but not a Jordan curve ~\cite{Be}. The precise statement of what we prove is the following.
\begin{theorem} Let $R$ be a  rational function such that its Newton map $N_R$ is quadratic. Then $N_R$ is conformally conjugate either to $N_1(z)=z-\frac{1}{\frac{d_1}{z}+\frac{d_2}{z-1}}$ or to $N_2(z)=z+\frac{1}{\frac{e_1}{z}+\frac{e_2}{z-1}}$ for some natural numbers $d_1, d_2, e_1, e_2$.
	\begin{enumerate}
		\item If $N_R$ is conformally conjugate to $N_1$ then the Fatou set of $N_R$ is the union of two completely invariant attracting domains and the Julia set is a Jordan curve.
		\item If $N_R$ is conformally conjugate to $N_2$ then the Fatou set  of $N_R$ is equal to a completely invariant attracting domain and the Julia set is totally disconnected.
	\end{enumerate}
	In other words, the Julia set of $N_R$ is either a Jordan curve or totally disconnected.
	\label{main} 
\end{theorem}
In \cite{MS}, Shishikura proved that if the Julia set of a rational function $f$ is disconnected then $f$ has at least two weakly repelling fixed points (i.e., whose multiplier is $1$ or is with absolute value bigger than  $1$). Since the Newton map $N_P$ of every polynomial $P$ has only one repelling fixed point, it follows from Shishikura's result that the Julia set of $N_P$ is connected. Since every pole of a rational function is a repelling fixed point of its Newton map (See Proposition~\ref{NFP}), there is no reason why the Julia set of such  a Newton map will be connected. Indeed, it is found to be false even for quadratic Newton maps. This is part of  Theorem~\ref{main}(2).
\par 
 By analyzing the nature of fixed points of Newton maps, Barnard et al. proved that no Newton map of any rational function is conformally conjugate to $z^3$~\cite{Conj}. As a significant generalization, we are able to prove that no Newton map of degree at least three is conformally conjugate to a unicritical polynomial (i.e., a polynomial with exactly one finite critical point). This is Theorem~\ref{unicritical} appearing in Section 4 of this article. This result is sharp in the sense that a cubic Newton map can be conformally conjugate to a polynomial with two finite critical points.
  The Julia sets of all such Newton maps are shown to be closed curves. More precisely, the following is proved.  
	\begin{theorem}
	If a cubic Newton map is conformally conjugate to a polynomial then its Fatou set  is $ A^*  \bigcup A$, where $A^*$ is the completely invariant attracting domain corresponding to a superattracting fixed point and $A$ is one of the following.
	\begin{enumerate} 
		
		\item $A$ is the union of two invariant attracting domains corresponding to  two finite attracting fixed points. The  Fatou set is the union of infinitely many components and each is simply connected. The Julia set is a self intersecting closed  curve. 
		\item $A$ is the completey invariant attracting domain corresponding to a finite attracting fixed point. In this case, the Julia set is a Jordan curve.
	\end{enumerate}
\label{cubicpoly}
\end{theorem}
The two possible Julia sets as stated in Theorem~\ref{cubicpoly}  are given in Figures~\ref{Jordan} and \ref{closedcurve}.
   
\par
  Section 2 discusses the nature of the fixed points of Newton maps. The Rational Fixed point Theorem and the Characterization theorem for Newton maps are presented in this section. Some basic facts about the Fatou set are also given in this section. All quadratic Newton maps  and their dynamics is dealt in Section 3. This section contains the proof of Theorem~\ref{main}. In Section 4, some observations are made on Newton maps of higher degree with special focus on the cubic case. All rational functions $R$ for which $N_R$ is cubic and is conjugate to a polynomial is given in  Table~\ref{table} with a detailed discussion in the last section.  
\par 
 Throughout this article, by saying two rational maps are conjugate we mean that they are  conformally conjugate. The composition of two functions $f, g$ is denoted by  $f \circ g$.
 
 \section{Preliminaries}
 
\subsection{Some basic properties of $N_R$}
Let $R(z)=\frac{P(z)}{Q(z)}$ be a rational function where $P$ and $Q$ are polynomials of degrees $d$ and $e$ respectively with no common factor. Let $P$ have $m$ distinct roots, $\alpha_i, i=1,2,...,m$, and the distinct roots of $Q$ be $\beta_j,j=1,2,...,n$. Then
$$P'(z)=P(z)\sum_{i=1}^{m}\frac{d_i}{z-\alpha_i}~ \text{ and }~
Q'(z)=Q(z)\sum_{j=1}^{n}\frac{e_j}{z-\beta_j}$$ where $d_i$ and $e_j$ are the multiplicities of $\alpha_i$ and $\beta_j$ respectively. Therefore $\sum_{i=1}^{m}d_i=d$, $\sum_{j=1}^{n}e_j=e$ and
\begin{align*}
R'(z)&=\frac{P'(z)Q(z)-P(z)Q'(z)}{Q^2(z)} \\
&=\frac{P(z)Q(z)\left[\sum_{i=1}^{m}\frac{d_i}{z-\alpha_i}-\sum_{j=1}^{n}\frac{e_j}{z-\beta_j}\right]}{Q^2(z)}\\
&=R(z)\left[\sum_{i=1}^{m}\frac{d_i}{z-\alpha_i}-\sum_{j=1}^{n}\frac{e_j}{z-\beta_j}\right].
\end{align*} This implies that

\begin{equation}\label{2.2}
N_R(z)=z-\frac{1}{\sum_{i=1}^{m}\frac{d_i}{z-\alpha_i}-\sum_{j=1}^{n}\frac{e_j}{z-\beta_j}}.
\end{equation}
The degree of $N_R$ can be expressed in terms of the number of distinct roots and poles of $R$. From Equation (\ref{2.2}), we get
$ \label{2.3}
\nonumber
N_R(z)= z-\frac{\prod_{i=1}^{m}(z-\alpha_i)\prod_{j=1}^{n}(z-\beta_j)}
{\tilde{A}-\tilde{B}}
$
where $\tilde{A}=\prod_{j=1}^{n}(z-\beta_j) \left[ \sum_{i=1}^{m} \{ d_i\prod_{k=1,k\neq i}^{m}(z-\alpha_k) \} \right] $ and \\
 $\tilde{B}= \prod_{i=1}^{m}(z-\alpha_i) \left[\sum_{j=1}^{n}\{ e_j\prod_{l=1,l\neq j}^{n}(z-\beta_l)\} \right]$. 
Consequently,
$$N_R (z) = \frac{(d-e-1)z^{m+n}+A_0z^{m+n-1}+...+A_{m+n-1}}{(d-e)z^{m+n-1}+B_0z^{m+n-2}+...+B_{m+n-2}}$$ for some complex numbers $A_0, A_1,A_2,..., A_{m+n-1}$ and $B_0,B_1, B_2,...B_{m+n-2}$.

Hence, we get
\begin{equation}\label{2.4}
deg(N_R)=\begin{cases}
m+n-1  &\mbox{if} \hspace{0.3 cm}  d=e+1 \\
m+n  &\mbox{if} \hspace{0.3 cm}   d\neq e+1 
\end{cases}. 
\end{equation}
 Note that a rational function of degree $k$ has $k+1$ fixed points counting multiplicities. 
 Now we determine all the fixed points of $N_R$ and their multipliers. 
 Since $\infty$ can also be a fixed point of an $N_R$, we need the following proposition dealing with such a situation (Page 41, ~\cite{Be}).
\begin{Proposition} \label{infty}
	Let $ f(z)=\frac{a_0z^k+a_1z^{k-1}+\dots +a_k}{b_0z^l+b_1z^{l-1}+\dots +b_l} $ be a rational function. 
	Then $\infty$ is a fixed point of $f$ if and only if $k > l.$
	 Moreover, the multiplier $\lambda$ of $\infty$ is given by 
	$$\lambda=
	\begin{cases}
	\frac{b_0}{a_0}  &\text{if } k=l+1\\
	0  & \text{if } k>l+1 .
	\end{cases} $$
\end{Proposition}
 The following result is a special case of Proposition $1$ in \cite{2} (for $\sigma=2$). However, we give a proof for completeness.
\begin{Proposition}\label{NFP}
Let $\{ \alpha_1, \alpha_2,...,\alpha_m \} $ and $\{ \beta_1, \beta_2,...,\beta_n \}  $ be the sets of all distinct finite roots and distinct finite poles of a rational function $R$ respectively. 
	
	\begin{enumerate}
		\item If $d_i$ is the multiplicity of $\alpha_i$ then  $\alpha_i$ is an attracting fixed point of $N_R$ with multiplier $\frac{d_i-1}{d_i}$.
	\item If  $e_j$ is the multiplicity of $\beta_j$  then  $\beta_j$ is a repelling fixed point of $N_R$ with multiplier $\frac{e_j+1}{e_j} $.
\item The point at $\infty$ is a fixed point of $N_R$ if and only if $d\neq e+1$ where $d=\sum_{i=1} ^m d_i$ and $e=\sum_{j=1} ^n e_j$, and its multiplier is $ \frac{d-e}{d-e-1}. $ Further, it is  attracting if $d \leq e$ (superattracting if $d=e$) and repelling if $d>e$.
	\end{enumerate}
\end{Proposition}
\begin{proof}
	From Equation (\ref{2.2}) it is clear that   $$N_R'(z)= 1-\dfrac{\sum_{i=1}^{m}\frac{d_i}{(z-\alpha_i)^2}-\sum_{j=1}^{n}\frac{e_j}{(z-\beta_j)^2}}{\left[\sum_{i=1}^{m}\frac{d_i}{z-\alpha_i}-\sum_{j=1}^{n}\frac{e_j}{z-\beta_j}\right]^2}.$$
	Let $A(z)=\prod_{i=1}^{m}(z-\alpha_i)$ and $B(z)=\prod_{j=1}^{n}(z-\beta_j)$. Then
	\begin{equation*}
	N_R'(z)=1-\dfrac{B(z)^2\sum_{i=1}^{m}\left[d_i \prod_{k=1,k\neq i}^{m}(z-\alpha_k)^2\right]-A(z)^2\sum_{j=1}^{n}\left[e_j\prod_{l=1,l\neq j}^{n}(z-\beta_l)^2\right]}{\left[B(z)\sum_{i=1}^{m}\{ d_i \prod_{k=1,k\neq i}^{m}(z-\alpha_k)\}-A(z)\sum_{j=1}^{n}\{e_j\prod_{l=1,l\neq j}^{n}(z-\beta_l)\} \right]^2}.
	\end{equation*}
	Hence, 
	\begin{align*}
	N_R'(\alpha_i)&=1-\dfrac{(B(\alpha_i))^2d_i\prod_{k=1,k\neq i}^{m}(\alpha_i-\alpha_k)^2}{(B(\alpha_i))^2 d_i^2\prod_{k=1,k\neq i}^{m}(\alpha_i-\alpha_k)^2}
	=\frac{d_i-1}{d_i} \qquad \text{for}~ i=1,2,...,m.
	\end{align*}
	Similarly, 
	\begin{align*}
	N_R'(\beta_j)&=1+\frac{(A(\beta_j))^2e_j\prod_{l=1,l\neq j}^{n}(\beta_j-\beta_l)^2}{(A(\beta_j))^2e_j^2\prod_{l=1,l\neq j}^{n}(\beta_j-\beta_l)^2}
=\frac{e_j+1}{e_j} \qquad \text{for}~ j=1,2,...,n.
	\end{align*}
	This completes the proof of $1$ and $2$.\\
	From Equation (\ref{2.4}) and Proposition \ref{infty}, it follows that $\infty$ is a fixed point of $N_R$ if and only if $d\neq e+1$, and its multiplier is  $ \frac{d-e}{d-e-1}$. Further, $\infty$ is attracting if $d \leq e$ (superattracting if $d=e$) and repelling if $d>e$. 
\end{proof}
 Every root of a polynomial $P $ with degree at least two, is an attracting fixed point of $N_P$. The point at $\infty$ is the only extraneous fixed point of $N_P$.  But for a rational function $R$, every finite pole turns out to be an extraneous fixed point of $N_R$. In order to make some further remarks on the above proposition, we need some definitions. The multiplicity of a fixed point of a rational function $f$ is its multiplicity as a root of $f(z)-z$. A fixed point is called simple if its multiplicity is $1$. This happens if and only if its multiplier is different from $1$~\cite{Mi}.
\begin{Remark}
For $d=e+1$, the degree of $N_R$ is $m+n-1$ and it has $m+n$ fixed points, namely $\alpha_i, i=1,2,...,m$ and $\beta_j, j=1,2,...,n$. Similarly for $d \neq e+1$, the degree of $N_R$ is $m+n$ and it has $m+n+1$ fixed points, namely $\alpha_i, i=1,2,...,m$, $\beta_j, j=1,2,...,n$ and $\infty$. In this case, $\infty$ is attracting if $d \leq  e$ and is repelling whenever $d >e$(in fact, $d > e+1$).
\end{Remark}
\begin{Remark} For every rational function $R$, the multiplier of each fixed point of $N_R$ is of the form $\frac{p}{q}$ where $p \in \mathbb{N} \bigcup \{0\}, q \in \mathbb{N}$ and $|p-q|=1$. In particular, all the fixed points of $N_R$ are simple. The multiplier is an integer only if it belongs to $\{0,2\}$. 
\label{multiplier}
\end{Remark}
 Conformal conjugacy preserves almost everything about the dynamics including the nature of fixed points.
We enumerate few such things  that are to be used repeatedly.
\begin{lemma}\cite{Be} Let $f$ and $g$ be two rational functions such that $f =\phi \circ g \circ \phi^{-1}$. Then, $\mathcal{F}(f)=\phi (\mathcal{F}(g))$ and $\mathcal{J}(f)=\phi (\mathcal{J}(g))$. Further,
\begin{enumerate} 
\item A point $z_0$ is a fixed point of $g$ if and only if $\phi(z_0)$ is a fixed point of $f $. Further, they have the same multiplier, i.e., $g '(z_0)= f'(\phi(z_0))$, and one is simple whenever the other is so. 
\item A Fatou component  $U$ of $g$ is an invariant attracting domain if and only if $\phi(U)$ is an invariant attracting domain of $f$. Further, $U$ is completely invariant if and only if $\phi(U)$ is so. 
\end{enumerate}
\label{conjugacy}
\end{lemma}

For a fixed point $z_0$ of a non-constant rational function $f$, the residue fixed point index, in short the residue index of $f$ at $z_0$ is defined by $$\iota(f,z_0)=\frac{1}{2\pi i}\oint \dfrac{dz}{z-f(z)}$$
where the integration is taken in a small loop around $z_0$ (in the sense that it does not surround any other fixed point of $R$) in the positive direction. If $z_0$ is a simple fixed point with multiplier $\lambda$ then $\iota (R,z_0)=\frac{1}{1-\lambda}$. The \textit{Rational Fixed Point Theorem} is a useful relation between   residue fixed point indices of all the fixed points of a rational function (See Theorem 12.4, \cite{Mi}).
\begin{theorem}[Rational Fixed Point Theorem]\label{hff} If $f$ is not the identity map and is a non-constant rational function then the sum of residue fixed point indices of all its fixed points in $\widehat{\mathbb{C}}$ is equal to $1$.
\end{theorem}
Since $\infty$ is a superattracting fixed point of every polynomial $P$ of degree at least two, the residue index  of $P$ at $\infty$ is $1$ and conseqently,  $\sum_{P(z)=z}\iota(P,z)=0$, where the sum is taken over all the finite fixed points of $P$. Two useful remarks follow.
\begin{Remark}
	The residue index of  $N_R$ at a fixed point  with multiplier $\frac{p}{q}$ is $q$ or $-q$; it is $q$ if and only if $p< q$. In other words, the residue index of a fixed point of a Newton map is positive if and only if the fixed point is attracting.  \label{index}
\end{Remark}
\begin{Remark}
The residue index  of a rational function at each fixed point is $q$ or $-q$ whenever its multiplier is $\frac{p}{q}$ for some  $p \in \mathbb{N} \bigcup \{0\}, q \in \mathbb{N}$ with $|p-q|=1$. Conversely, if the residue index of a simple fixed point is a nonzero integer $q$ then its multiplier is of the form  $p \in \mathbb{N} \bigcup \{0\}, q \in \mathbb{N}$ with $|p-q|=1$.  
\label{multiplier-index}
\end{Remark} 
As an important application of the Rational Fixed Point Theorem, Buff and Heneriksen proved the following.
\begin{lemma}\cite{2}\label{Buff}
Let all the fixed points of two rational functions $f_1$ and $f_2$ be simple. If $f_1$ and $f_2$ have the same set of fixed points $ \{z_1, z_2, ...,z_n\}$ such that  $f_1'(z_i)=f_2'(z_i)$ for all $i$ then $f_1(z)=f_2(z)$ for all $z \in \widehat{\mathbb{C}}$.
\end{lemma}
This lemma leads to a useful result for quadratic rational functions.
\begin{lemma}
Let all the fixed points of two quadratic rational functions $f_1$ and $f_2$  be simple. Then  $f_1$ is conjugate to $f_2$ if and only if the set of all the multipliers of all fixed points of $f_1$ is the same as that of $f_2$.
\label{quadratic}
\end{lemma}
\begin{proof} 
Since all the fixed points of $f_1$ and $f_2$ are simple, there are three distinct fixed points for each $f_k, k=1,2$.
Let $z_i$ and $w_i$ be the fixed points of $f_1$ and $f_2$ respectively  for $i=1,2,3$. 
\par If the set of all the multipliers of all  fixed points of $f_1$ is the same as that of $f_2$ then possibly after rearranging $z_i$'s, we get that  $f_1 '(z_i)=f_2 '(w_i)$ for $i=1,2,3$. Considering the M\"{o}bius map $\phi$ with $\phi(z_i)=w_i, i=1,2,3$, it is observed that $\phi^{-1} \circ f_2 \circ \phi$ fixes $z_i$ with multiplier $f_2 '(w_i)$ for each $i$. Now it follows from Lemma~\ref{Buff} that $f_1 = \phi^{-1} \circ f_2 \circ \phi $. In other words, $f_1$ is conjugate to $f_2$. Conversely if $f_1$ is conjugate to $f_2$ then for some M\"{o}bius map $\psi$, $f_1=\psi^{-1} \circ f_2 \circ \psi$. Further, $f_1$ fixes $z_i$ if and only if $f_2 $ fixes $\psi(z_i)$, which is an element of  $\{w_i: i=1,2,3\}$. Without loss of generality assume that  $\psi(z_i)=w_i, i=1,2,3$. Now $f_1 '(z_i)=f_2 '(w_i)$ for $i=1,2,3$ by Lemma~\ref{conjugacy} and the proof is complete.
\end{proof}
Lemma~\ref{Buff} asserts that a rational function with only simple fixed points is uniquely determined by its fixed points and their multipliers. This fact leads to a characterization of Newton maps. Such a characterization is given by Crane (See Lemma $4.1$,  \cite{EC}). We present this with a minor change in statement and with a quite straightforward proof.
\begin{theorem}[Characterization of Newton maps]
A rational map $N$ of degree at least two is the Newton map of a rational function $R$ if and only if all the fixed points of $N$ are simple and all but one of their multipliers are of the form $ \frac{p}{q} $ for some $p \in \mathbb{N} \bigcup \{0\}, q \in \mathbb{N}$ with $|p-q|=1$. Moreover, each finite fixed point of $N$ with multiplier $\frac{p}{q}$ is either a root (if $p < q$) or a pole (if $p> q$) of $R$ with multiplicity $q$.
\end{theorem}
\begin{proof}
All the fixed points of the Newton map of a rational function  are simple and their mutipliers are of the form $\frac{p}{q}$ for some $p \in \mathbb{N} \bigcup \{0\}, q \in \mathbb{N}$ with $|p-q|=1$. This follows from Proposition~\ref{NFP}.

\par Conversely, let all the fixed points of a rational map $N$ be simple and  all but one of their multipliers are of the form  
$\frac{p}{q} $ for some $p \in \mathbb{N} \bigcup \{0\}, q \in \mathbb{N}$ with $|p-q|=1$. By the Rational Fixed Point Theorem, the  residue index of the remaining fixed point is a nonzero integer and its multiplier is of the same form by Remark~\ref{multiplier-index}.
Let $\{\alpha_i,~ \beta_j: ~ i=1,2,...,m \text{ and } j=1,2,...,n\}$ be the set of all finite fixed points of $N$ such that $N'(\alpha_i)=\frac{d_i-1}{d_i}, ~ d_i\in \mathbb{N}$ for $i=1,2,...,m$ and $N'(\beta_j)=\frac{e_j+1}{e_j},~ e_j\in \mathbb{N}$ for $j=1,2,...,n$. Then considering the rational function $R(z)=\frac{\prod_{i=1}^{m}(z-\alpha_i)^{d_i}}{\prod_{j=1}^{n}(z-\beta_j)^{e_j}}$ it is observed that its Newton map $N_R$ has finite fixed points $\alpha_i$ and $\beta_j$ with multipliers $ \frac{d_i-1}{d_i} $ and $\frac{e_j+1}{e_j}$ respectively for $i=1,2,...,m$ and $j=1,2,...,n.$ If $N(\infty)\neq \infty$ then  $N=N_R$ by Lemma \ref{Buff}.	If $N(\infty) = \infty$ then $\sum_{i=1}^{m}d_i \neq 1+ \sum_{j=1}^{n}e_j$  and the multiplier of $\infty$ is $\frac{\sum_{i=1}^{m}d_i - \sum_{j=1}^{n}e_j}{\sum_{i=1}^{m}d_i - \sum_{j=1}^{n}e_j-1}$, by Theorem~\ref{hff}. Note that in this case, $N_R (\infty) =\infty$ and its multiplier is $\frac{\sum_{i=1}^{m}d_i - \sum_{j=1}^{n}e_j}{\sum_{i=1}^{m}d_i - \sum_{j=1}^{n}e_j-1}$, by Proposition~\ref{NFP}. Therefore, by Lemma \ref{Buff}, $N=N_R$.
\par The rest is obvious.	
\end{proof}

\subsection{More on Fatou components}
A point $z_0$ is called a $p-$periodic point of a rational function $f$ if $p$ is the least natural number such that $f^p (z_0)=z_0$. This is attracting, repelling or indifferent if it is so as a fixed point of $f^p$. A Fatou component $U$ is called $p-$periodic if $f^p(U)=U$. In this case, we say $\{U, U_1, U_2,..., U_{p-1}\}$ is  a $p-$cycle of Fatou components. It is known that every Fatou component of each rational function is ultimately periodic. A periodic Fatou component is called an attracting(superattracting) domain or a parabolic domain if it contains an attracting(superattracting) periodic point or its boundary contains a rationally indifferent periodic point respectively. The only other possibility for a $p-$periodic Fatou component $U$  is that on $U$, $f^p$ is conformally conjugate to an irrational rotation of the unit disk or of the annulus $\{z: 1<|z|<r\}$ for some $r$. In the first case, $U$ is called a Siegel disk whereas it is called a Herman ring in the second situation.  A Fatou component $U$ is called  invariant if it is $1-$periodic. 
\par The attracting basin $A(z_0)$ of an attracting fixed point $z_0$ is the set $\{z \in \widehat{\mathbb{C}}: f^n (z) \to z_0~\mbox{as}~ n \to \infty\}$. Its connected component containing $z_0$ is the invariant attracting domain and is usually referred as the immediate basin of attraction of $z_0$. The attracting basin is not necessarily completely invariant.  All these facts and a good introduction to the subject can be found in the book~\cite{Be}.
\par
A point $c$ is called a critical point of a rational function $f$ if the local degree of $f$ at $c$ is at least two. This is equivalent to $f'(c)=0$ whenever $c \in \mathbb{C}$. Critical points are known to control the dynamics. The precise statement follows.
\begin{lemma}[\cite{Be}]
Let $f$ be a rational function with degree at least two with a $p-$cycle of Fatou components  $\mathcal{C} = \{U_0,U_1, U_2,...,U_{p-1} \}$.
\begin{enumerate}
\item If $\mathcal{C}$ is a cycle of attracting domains  or parabolic domains then  $\bigcup_{i=0}^{p-1} U_i$ contains at least one critical point of $f$.
\item If $\mathcal{C}$ is a cycle of Siegel discs or Herman rings then  the union of boundaries $\bigcup_{i=0}^{p-1} \partial U_i$ is contained in the closure of $\bigcup_{n=0}^{\infty} f^n (C)$ where $C$ is the set of all critical points of $f$.
\end{enumerate}
\label{basic}
\end{lemma}

\section{Quadratic Newton maps}
%
%
%
%
%

The following theorem proves that there are exactly two quadratic Newton maps up to conjugacy.
 \begin{theorem}
 	Let $N_1$ and $N_2$ be the Newton maps of $R_1(z)=z^{d_1}(z-1)^{d_2}$ and $R_2(z)=\frac{1}{z^{e_1}(z-1)^{e_2}}$ respectively.
 	 Then every quadratic Newton map is conjugate either to $N_{1} $ or to $N_{2}$ for some $d_1, d_2, e_1, e_2 \in \mathbb{N}$. 
 	 \label{2quadratic-Newtonmaps}
 \end{theorem}
\begin{proof}
	Note that $N_{1}$ fixes $0,1$ and $\infty$ with multipliers $\frac{d_1 -1}{d_1}, \frac{d_2 -1}{d_2}$ and  $\frac{d_1 +d_2}{d_1+d_2-1}$ respectively. Similarly $N_{2}$ fixes  $0,1$ and $\infty$ with multipliers $\frac{e_1 +1}{e_1}, \frac{e_2+1}{e_2}$ and $\frac{-e_1-e_2}{-e_1-e_2-1}$ respectively.
	\par 
Since $N_R =N_{aR}$ for every nonzero complex number $a$, it is enough to consider only those rational functions $R=\frac{P}{Q}$ where both $P$ and $Q$ are monic polynomials. Also without loss of generality it is assumed that $P$ and $Q$ have no common factor. Formula (\ref{2.4}) in Section 2 gives that for a given rational function $R$, its Newton map $N_R$ is quadratic if $d\neq e+1$ and $m+n=2$, or $d=e+1$ and $m+n=3$. The proof will be complete by showing that the set of multipliers of all the fixed points of $N_R$ is the same as that of $N_1$ or $N_2$  in each case and then by using Lemma~\ref{quadratic}. \\
\underline{\textit{Case- I :$~d\neq e+1$ and $m+n=2$:}}
\begin{itemize}
	\item $ m=2, n=0: $ $ R(z)=(z-\alpha_1)^{d_1}(z-\alpha_2)^{d_2} $ where $\alpha_1, \alpha_2\in \mathbb{C}$ are distinct and $d_1, d_2\in \mathbb{N}$.  Its Newton map $N_R$ fixes $\alpha_1, \alpha_2$ and $\infty$ with multipliers $\frac{d_1 -1}{d_1}, \frac{d_2 -1}{d_2}$ and  $\frac{d_1 +d_2}{d_1+d_2-1}$ respectively. This is conjugate to $N_1$ by Lemma~\ref{quadratic}.
	\item  $m=0, n=2:$ $R(z)=\frac{1}{(z-\beta_1)^{e_1}(z-\beta_2)^{e_2}}$ where $\beta_1,\beta_2\in \mathbb{C}$ are distinct and $e_1, e_2\in \mathbb{N}$. The Newton map $N_R$ fixes $\beta_1, \beta_2$ and $\infty$ with multipliers $\frac{e_1 +1}{e_1}, \frac{e_2+1}{e_2}$ and $\frac{-e_1-e_2}{-e_1-e_2-1}$ respectively. This is conjugate to $N_2$ by Lemma~\ref{quadratic}.
	\item $ m=1,n=1: $ $ R(z)=\frac{(z-\alpha)^{d}}{(z-\beta)^{e}} $ where $\alpha, \beta \in \mathbb{C}$ are distinct and $d, e\in \mathbb{N}$.  The Newton map $N_R$ fixes $\alpha, \beta$ and $\infty$ with multipliers $\frac{d -1}{d}, \frac{e+1}{e}$ and $\frac{d-e}{d-e-1} $ respectively. If $d \leq e$ then considering $d_1=d$ and $d_2=e-d+1$ it is seen that the set of multipliers of all the fixed points of $N_R$ is the same as that of $N_1$. Therefore, $N_R$ is conjugate to $N_1$ for $d \leq e$ by Lemma~\ref{quadratic}. Similarly for $d > e$, take $e_1=e $ and $ e_2=d-e-1$. Observe that the set multipliers of all the fixed points of $N_R$ is the same as that of $N_2$. Hence, $N_R$ is conjugate to $N_2$ whenever $d \geq e$ by Lemma~\ref{quadratic}.
	
\end{itemize}
\underline{\textit{Case- II:~  $d=e+1$ and $m+n=3$ :}} Note that none of $m,n$ can be zero in this case.
\begin{itemize}
	\item $ m=1, n=2: $ $ R(z)=\frac{(z-\alpha)^{e_1+e_2+1}}{(z-\beta_1)^{e_1}(z-\beta_2)^{e_2}} $ where $\alpha, \beta_1, \beta_2\in \mathbb{C}$ are distinct and $e_1, e_2\in \mathbb{N}$. The Newton map $N_R$ fixes $\alpha, \beta_1$ and $\beta_2$ with multipliers $ \frac{e_1 +e_2}{e_1 +e_2+1}, \frac{e_1 +1}{e_1}$ and $ \frac{e_2 +1}{e_2}$ respectively. This $N_R$ is conjugate to $N_2$ by Lemma~\ref{quadratic}.
	\item $ m=2, n=1: $ $ R(z)=\frac{(z-\alpha_1)^{d_1}(z-\alpha_2)^{d_2}}{(z-\beta)^{d_1+d_2-1}}. $ where $\alpha_1, \alpha_2,\beta\in \mathbb{C}$ are distinct and $d_1, d_2\in \mathbb{N}$. The Newton map $N_R$ fixes $\alpha_1, \alpha_2$ and $\beta$ with multipliers $\frac{d_1 -1}{d_1}, \frac{d_2 -1}{d_2}$ and $\frac{d_1 +d_2}{d_1 +d_2-1}$. This $N_R$ is conjugate to $N_1$ by Lemma~\ref{quadratic}.
\end{itemize}
 
\end{proof}
\begin{Remark} Every quadratic Newton map has three distinct fixed points. By Remark~\ref{multiplier-index}, each of their  residue  indices is a nonzero integer. If a quadratic Newton map is conjugate to a polynomial then one  residue index is $1$ and the other two are of opposite signs, by the Rational Fixed Point Theorem.
	\end{Remark}
\begin{Remark}	The map  $N_1(z)=
	N_{d_1, d_2} (z)=\frac{(d_1 + d_2 -1)z^2 +(1- d_1)z}{(d_1+d_2)z-d_1}$ has two attracting fixed points, namely $0$ and $1$, and the third fixed point $\infty$ is repelling. Similarly, $N_2 (z)=N_{e_1, e_2}(z) =\frac{(e_1 + e_2 +1)z^2 +(-1- e_1)z}{(e_1 +e_2)z-e_1}$ has two repelling fixed points, namely $0$ and $1$, and the third fixed point $\infty$ is attracting.
\end{Remark}

Barnard et al. \cite{Conj} found all possible forms of $R$ for which $N_R$ is conjugate to a polynomial. Here we simplify their result. A well-known result is used to do it here.
\begin{lemma}\label{CP}
A quadratic rational map is conjugate to a polynomial if and only if it has a superattracting fixed point.  
\end{lemma} 
The following result (Theorem 4.1.2,\cite{Be}) is used to prove this lemma. A point $z_0$ is called an exceptional point of a rational function $f$ if its grand orbit  $[z]=\{z: f^k(z)=z_0~\mbox{ for some }~k > 0 \} \cup\{f^k(z_0): k \geq 0 \}$ is finite.
\begin{lemma}\label{Be}
	A non-constant rational function $f$ is conjugate to a polynomial if and only if there is an exceptional point $w\in \widehat{\mathbb{C}}$ such that $[w]$ is singleton.
	
\end{lemma} 
\begin{proof}[Proof of Lemma \ref{CP}]
	Let $f$ be a quadratic rational function with a superattracting fixed point $w_0$. If there is a point $z_0$ different from $w_0$ such that $f(z_0)=w_0$ then there exists a neighborhood of $w_0$, say $N_{0}$ such that $f^{-1}(N_{0})$ has two components, one $\Omega_1$, containing $w_0$ and other $\Omega_2$, containing $z_0$. However, $f:\Omega_1 \to N_{0}$ is a proper map of degree two (since $w_0$ is critical) and every point of $N_{0}$ has at least one pre-image in $\Omega_2$. This gives that each point in $N_{0}-\{w_0\}$ has at least three pre-images, which contradicts the assumption that $f$ is quadratic. Hence $w_0$ is the only pre-image of itself under $f$, and from Lemma \ref{Be}, we conclude that $f$ is conjugate to a polynomial.
	\par 
	Conversely, if a quadratic rational map $f$ is conjugate to a polynomial $P$, i.e., $ (\phi \circ f \circ \phi^{-1})(z)=P(z) $ for some M\"{o}bius map $\phi$, then $f(\phi^{-1}(\infty))=\phi^{-1}(\infty)$ and  its multiplier is $0$. This gives that $\phi^{-1}(\infty)$ is a superattracting fixed point of $f$ by Lemma~\ref{conjugacy}.
\end{proof}
Here is the desired description of all rational functions whose Newton maps are quadratic and are conjugate to a polynomial.
\begin{theorem}\label{Poly}
	A quadratic Newton map $N_R$ of a rational function $R$ is conjugate to a polynomial if and only if $R$ is one of the following.
	\begin{enumerate}
		\item $ R(z)=(z-\alpha_1)(z-\alpha_2)^{d_2} $ where $\alpha_1, \alpha_2\in \mathbb{C}$ are distinct and $d_2\in \mathbb{N}$.
		\item $ R(z)=\frac{(z-\alpha_1)}{(z-\beta_1)^{e_1}} $ where $\alpha_1, \beta_1\in \mathbb{C}$ are distinct and $e_1\in \mathbb{N}$.
		\item $ R(z)=\frac{(z-\alpha_1)(z-\alpha_2)^{d_2}}{(z-\beta_1)^{d_2}}. $ where $\alpha_1, \alpha_2,\beta_1\in \mathbb{C}$ are distinct and $d_2\in \mathbb{N}$.
		\item $ R(z)=\frac{(z-\alpha_1)^{d}}{(z-\beta_1)^{d}} $ where $\alpha_1, \beta_1\in \mathbb{C}$ are distinct and $d\in \mathbb{N}$.
	\end{enumerate}
\end{theorem}
\begin{proof}
	The map $N_R$ is conjugate to a polynomial if and only if it has a superattracting fixed point, by Lemma~\ref{CP}. This amounts to a fixed point with multiplier $0$. This corresponds to one of the following cases.
	\par If $ d\neq e+1 $ then $ R(z)$ is $(z-\alpha_1)(z-\alpha_2)^{d_2} $ for some distinct $\alpha_1, \alpha_2\in \mathbb{C}$ and $d_2\in \mathbb{N}$ or is 
		 $ \frac{(z-\alpha_1)}{(z-\beta_1)^{e_1}} $ for some distinct $\alpha_1, \beta_1\in \mathbb{C}$ and $e_1\in \mathbb{N}$. In these cases, $\alpha_1$ is the superattracting fixed point of $N_R$. The only other possible form of $R$ in this case is $\frac{(z-\alpha_1)^{d}}{(z-\beta_1)^{d}}$ for some distinct $\alpha_1, \beta_1\in \mathbb{C}$ and $d \in \mathbb{N}$. Here $\infty$ is the desired superattracting fixed point. 
\par  
If $ d= e+1 $ then $ R(z)=\frac{(z-\alpha_1)(z-\alpha_2)^{d_2}}{(z-\beta_1)^{d_2}}$ for some distinct $\alpha_1, \alpha_2,\beta_1\in \mathbb{C}$ and $d_2\in \mathbb{N}$, where $\alpha_1$ is the superattracting fixed point. 
\end{proof}
  
The Fatou set and the Julia set of all quadratic Newton maps are to be determined now. We use some known results. 
\begin{lemma}[\cite{brolin}]
	If a rational function $f$ of degree at least two has two completely invariant attracting domains containing all its critical points then its Julia set is a Jordan curve.
	\label{Jordancurve} 
\end{lemma}
\begin{lemma}[\cite{Be}]
	If a rational function $f$ of degree at least two has an  invariant attracting domain containing all the critical points of $f$  then its Julia set is  totally disconnected.
	\label{cantor} 
\end{lemma}
We now present the proof of the main theorem.
\begin{proof}[Proof of Theorem\ref{main}]
	In view of Theorem~\ref{2quadratic-Newtonmaps} and Lemma~\ref{conjugacy}, it is enough to determine the dynamics of the maps $N_1$ and $N_2$. Note that each such map has at most two distinct critical points.
	\par
	There are two attracting fixed points of $N_1$ and each corresponding attracting domain contains a critical point by Lemma~\ref{basic}. The map $N_1$ restricted to each of these attracting domains is a proper map of degree at least two. In fact, the degree is exactly two as $N_R$ is quadratic.  In other words, each such attracting domain is completely invariant. The Julia set is a Jordan curve by Lemma~\ref{Jordancurve}.
	\par 
	 The map $N_2$ has a single attracting fixed point, namely $\infty$. There is a critical point in the  attracting domain corresponding to $\infty$ by Lemma~\ref{basic}. 
	Note that $N_2 (z) =\frac{(e_1 + e_2 +1)z^2 +(-1- e_1)z}{(e_1 +e_2)z-e_1}$ for some natural numbers $e_1, e_2$ giving that $N_2 (\overline{z})=\overline{N_2(z)}$ for all $z$. Further,  the critical points of $N_2$, being the roots of a quadratic polynomial with real coefficients,  are complex conjugates of each other.  Therefore the other critical point also belongs to the attracting domain. By Lemma~\ref{cantor}, the Julia set of $N_2$ is totally disconnected.
	    
\end{proof}
\section{Cubic and higher degree Newton maps}
Barnard et al. proved that there does not exist any rational function whose Newton map is conjugate to $z^3$. We prove this for a larger class of polynomials, namely those with a single finite critical point. Such polynomials are called unicritical polynomials and are of the form $a (z-\alpha)^ \sigma + b$ for some $a, \alpha, b \in \mathbb{C}, a \neq 0$ and $\sigma  \in \mathbb{N}$.
\begin{theorem}
 There is no rational function whose Newton map is conjugate to $a (z-\alpha)^ \sigma + b$ for any $a, \alpha, b \in \mathbb{C}, a \neq 0$ and for any integer $\sigma\geq 3$.
\label{unicritical}
 \end{theorem}
 \begin{proof}
 For every $a, \alpha, b \in \mathbb{C}, a\neq 0$, $P(z)=a(z-\alpha)^ \sigma + b$ is conjugate to $z^\sigma + c$ where $c=\frac{b- \alpha}{\mu}$. In fact, considering $\phi(z)=\mu z + \alpha $ for $ \mu^{\sigma -1}=\frac{1}{a}$, it is observed that  $a (\phi(z) - \alpha)^\sigma +b =\phi(z^\sigma )+b- \alpha$. In other words, $(\phi^{-1} \circ P \circ \phi)(z)=z^\sigma + \frac{b -\alpha}{\mu}$.
 		
 \par Suppose on the contrary that there is a rational function $R$ whose Newton map $N_R$ is conjugate to $a(z-\alpha)^ \sigma + b$ for some $a, \alpha, b \in \mathbb{C}, a\neq 0$ and $\sigma \geq 3$. Then, by the observation made in the previous paragraph, $N_R$ is conjugate to $z^\sigma+c$ for some complex number $c$. If $z_i$ is a finite fixed point  of  $z^\sigma+c$ then its multiplier is $\sigma z_i ^{\sigma-1}$. It follows from Remark~\ref{multiplier} and Lemma~\ref{conjugacy} that  $\sigma z_i ^{\sigma-1} =\frac{p_i}{q_i}$ for some $p_i \in \mathbb{N} \bigcup \{0\}, q_i \in \mathbb{N} $ and $|p_i-q_i|=1$. This gives that $z_i ^{\sigma-1} =\frac{p_i}{\sigma q_i}$ for each fixed point $z_i$ of $z^\sigma +c$. Using this and noting that $z_i$ satisfies $z^\sigma -z +c=0$, we have $z_i =\frac{c \sigma q_i}{\sigma q_i -p_i}$ for each $i$. Since all the fixed points of $z^\sigma +c$ are distinct by  Remark~\ref{multiplier} and Lemma~\ref{conjugacy}, there are $\sigma$ distinct solutions of $z^\sigma -z +c=0$ and  $z^\sigma -z +c=\prod_{i=1}^\sigma (z-z_i)$. Comparing the coefficients of $z^{\sigma -1}$ on both the sides, it is found that $\sum_{i=1}^{\sigma} z_i =0$. Consequently, $c \sigma \sum_{i=1}^{\sigma} \frac{q_i}{\sigma q_i -p_i}=0$. Since $p_i \in \mathbb{N} \bigcup \{0\}, q_i \in \mathbb{N} $ and $|p_i-q_i|=1$ and $\sigma \geq 3$, $\frac{q_i}{\sigma q_i -p_i} >0$ for each $i$ and consequently, $c=0$. Hence $N_R$ is conjugate to $z^{\sigma}$. The proof will be completed by obtaining a contradiction to this statement. 
 		 \par 
 		The function $z^\sigma$ has a fixed point with multiplier $\sigma$. In fact, each $\frac{1}{\sigma -1}$-th root of unity is a fixed point of $z^\sigma$ with multiplier $\sigma$. By Lemma~\ref{conjugacy}, $N_R$ has a fixed point with multiplier $\sigma$. However, this is not possible by Remark~\ref{multiplier} as $\sigma \geq 3$. 
    	\end{proof}
 
 
 As stated earlier, Newton maps can be conjugate to polynomials. Theorem~\ref{cubicpoly} deals with the cubic case. We now present  its proof. 
  
\begin{proof}[Proof of Theorem~\ref{cubicpoly}]
	
Let a cubic Newton map $N_R$ be conjugate to a  polynomial $\tilde{P}$. Then $\tilde{P}(z)=a_3 z^3 +a_2 z^2 +a_1 z+a_0$ for $a_3 \neq 0$. Considering $\psi (z)=\frac{z}{\sqrt{a_3}}-\frac{a_2}{3 a_3}$, it can be seen that $\psi^{-1} \circ \tilde{P} \circ \psi$ is $ z^3+a z +b$ for some $a, b \in \mathbb{C}, a \neq 0$. Therefore, $N_R$ is conjugate to   $P(z) =z^3+a z +b$. In other words, $ N_R= \phi \circ P \circ \phi^{-1}$ for some M\"{o}bius map $\phi$.  In view of Lemma~\ref{conjugacy}, it is sufficient to prove this theorem for $P$. It is well known that the point $\infty$ is a superattracting fixed point of $P$ and  the corresponding attracting domain is completely invariant. This is the required $A^*$. Note that $\infty$ is a critical point of $P$ and is in $A^*$. We assert that the Fatou set of $P$ either contains another completely invariant attracting domain or two attracting domains. 
\par Since all the fixed points of $N_R$ are simple, all the fixed points of $P$ are simple by Lemma~\ref{conjugacy}. Let the finite fixed points of $P$ be $z_1, z_2, z_3$. Then the multiplier of $z_i$ is $3 z_i ^2 +a$ and its  residue index is $\frac{1}{1- 3 z_i ^2 -a}$. Note that the  residue index of $N_R$  at each of its fixed points is a nonzero integer and it is $1$ at superttracting fixed points. Since conjugacy preserves multipliers and hence  residue indices, $\frac{1}{1- 3 z_i ^2 -a}$  is an integer, say $n_i$ for each $i$.  Since each fixed point $z_i$ is a root of $z^3+(a-1)z +b$, $z_1 +z_2 +z_3=0, z_1 z_2 +z_1 z_3 +z_2 z_3 =a-1$ and $-z_1 z_2 z_3=b$. Note that $3 z_i ^2 +a=1-\frac{1}{n_i}$ and $3(z_1 ^2 + z_2 ^2 + z_3 ^2 )+ 3a=3-(\frac{1}{n_1}+\frac{1}{n_2}+\frac{1}{n_3})$. 
But $z_1 ^2 + z_2 ^2 + z_3 ^2=(z_1 +z_2 +z_3)^2 -2(z_1 z_2 +z_1 z_3 +z_2 z_3 )$. This gives that $a=1+\frac{1}{3}(\frac{1}{n_1}+\frac{1}{n_2}+\frac{1}{n_3}) $. It follows from the Rational Fixed Point Theorem  that  $n_1 +n_2 +n_3 =0$, which gives that at least one of these $n_i's$ is positive and one is negative. Let $n_1 > 0$ and $n_2 < 0$. By Remark~\ref{index}, an $n_i$ is positive if and only if the corresponding fixed point $z_i$ is attracting. Thus our assumtion means that $z_1$ is attracting and $z_2$ is repelling. The  residue  index $n_3$ of $z_3$ can be positive or negative giving rise to two cases.
\par Let $n_3 >0$. Then $P$ has three attracting fixed points, namely $z_1, z_3$ and $\infty$. For $i=1, 3$, the attracting domain corresponding to $z_i$ contains at least one ciritical point by Lemma~\ref{basic}. In this situation $P$ has  two finite critical points ruling out any Fatou component different from the attracting domains corresponding to $z_1, z_3$ and $\infty$ by Lemma~\ref{basic}.  Therefore, the Fatou set of $P$ is the union of $A^*$ and the invariant attracting domains of $z_1$ and that of $z_3$.  Since $A^*$ does not contain any finite critical point, it is simply connected, by Theorem 9.5.1~\cite{Be}. Complete invariance of $A^*$ gives that all other Fatou components are simply connected. Since there are at least three Fatou components, the number of Fatou components is infinity (see Theorem 5.6.1~\cite{Be}). Each critical point is attracted to some attracting fixed point. It follows from  Lemma 9.9.1~\cite{Be} that the Julia set of $P$ is a closed curve. Clearly, it is self intersecting.

\par For $n_3 < 0$, $P$ has two attracting fixed points $z_1, \infty$ and two repelling fixed points $z_2$ and $z_3$. Note that $n_2, n_3 <0, n_1 =-n_2 -n_3$ and therefore $\frac{1}{n_1}+ \frac{1}{n_2}+\frac{1}{n_3} < 0$. Further $ -2 \leq  \frac{1}{n_2}+\frac{1}{n_3} <0 $ gives that $ -2 \leq \frac{1}{n_1}+ \frac{1}{n_2}+\frac{1}{n_3} < 0$. This gives that $\frac{1}{3}  \leq a <1$. Since $a>0$, the two finite critical points of $P$ are purely imaginary and are complex conjugates of each other. Further, $3 z_i ^2 +a=1-\frac{1}{n_i}, i=2,3$ gives that $z_2$ and $ z_3$ are real numbers. Now $z_1 +z_2 +z_3 =0$ implies that $z_1$ is also real. Therefore, $b$, which is $-z_1 z_2 z_3$ is real. Consequently, $P(\overline{z})=\overline{P(z)}$ for all $z$. The attracting domain $A(z_1)$ of $z_1$ contains a critical point $c$ and hence the other critical point $\overline{c}$. Since $A^*$ is completely invariant, this attracting domain is simply connected. Further, $A(z_1)$ contains both the finite critical points of $P$. The degree of $P$ on $A(z_1)$ must be three by the Riemann-Hurwitz formula (Theorem 5.5.4~\cite{Be}).  In other words, $A(z_1)$ is completely invariant and the Fatou set of $P$ is the union of two completely invariant attracting domains and the Julia set is a Jordan curve by Lemma~\ref{Jordancurve}.
\end{proof}

\begin{figure}[h!]
	\begin{subfigure}{.5\textwidth}
		\centering
		\includegraphics[width=1\linewidth]{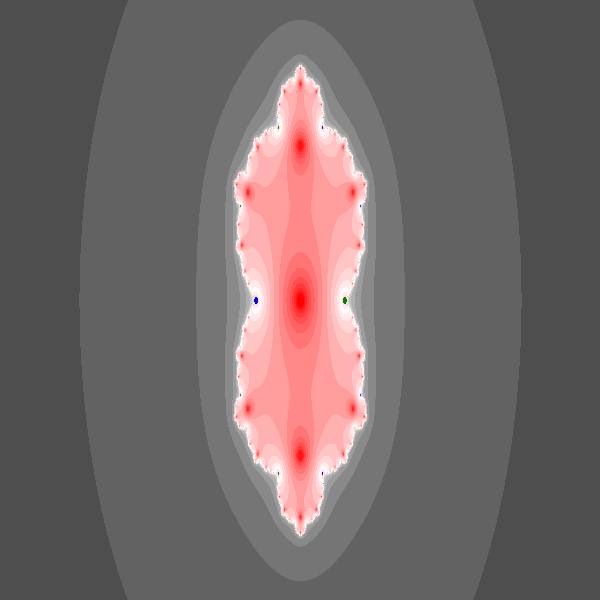}
		\caption{The Julia set of $z^3 +\frac{3}{4}z$, the Newton map of $ \frac{z^4}{(z-0.5)^2 (z+0.5)^2}$ }
		\label{Jordan}
	\end{subfigure}\hspace{0.5cm}
	\begin{subfigure}{.5\textwidth}
		\centering
		\includegraphics[width=1\linewidth]{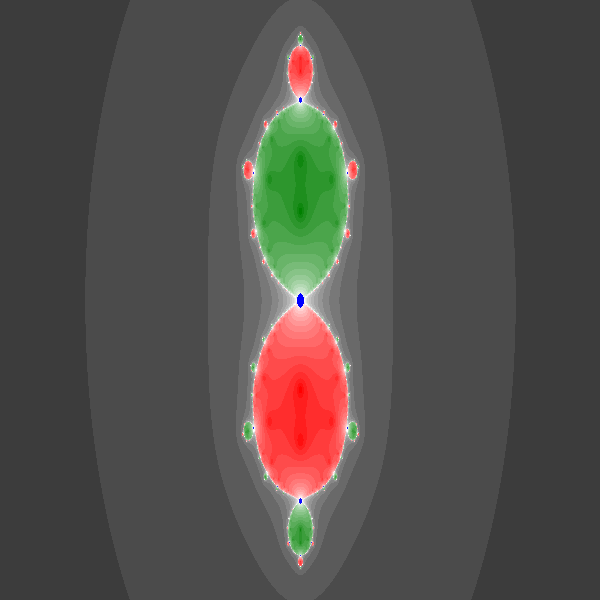}
		\caption{The Julia set of $z^3 +\frac{5}{4}z$, the Newton map of $\frac{(z+0.5 i)^2 (z-0.5 i)^2}{z^4}$}
		\label{closedcurve}
	\end{subfigure}
\end{figure}

  \begin{example}
  The Newton map of $ \frac{z^4}{(z-0.5)^2 (z+0.5)^2}$ is $z^3 +\frac{3}{4}z$ and it has only one finite attracting fixed point, namely $0$. Its Julia set is a Jordan curve. This is the boundary of the attracting domain of $\infty$, the black region as seen in Figure~\ref{Jordan}. The Newton map of $\frac{(z+0.5 i)^2 (z-0.5 i)^2}{z^4}$ is $z^3 +\frac{5}{4}z$, which has two finite attracting fixed points, namely $0.5 i$ and $-0.5i$. The corresponding attracting domains are seen in green and red which are not completely invariant. The blue dot at the intersection of two invariant attracting domains is the repelling fixed point $0$. The Julia set is the boundary of the attracting basin of $\infty$, the black region which is a self intersecting closed curve. This is Figure~\ref{closedcurve}. The figures are generated using polynomiography software.
  \end{example}
  \section{Concluding remarks}
  Firstly, we describe all the rational functions whose Newton maps are cubic and determine the situations when such a Newton map is conjugate to a polynomial.
  \par 
The point at $\infty$ is an exceptional point of every polynomial and it has another exceptional point if and only if it is a monomial  (See Theorem 4.1.2~\cite{Be}). Note that every monomial is unicritical.
If $N_R$ is the Newton map of a rational function $R$ and it is conjugate to a polynomial then this polynomial can not be unicritical by Theorem~\ref{unicritical}.  and therefore   $N_R$ has exactly one exceptional point. 
 Note that an exceptional point is necessarily a superattracting fixed point of a cubic Newton map and its local degree at such a point is three. Further, every finite superattracting fixed point of $N_R$ is a simple root of $R$. 
 
%
Two cases arise depending on the value of $N_R (\infty)$ . We need following observations for discussing these cases.
\begin{obs}
	\begin{enumerate}
		\item If $0$ is exceptional for a cubic Newton map $\frac{Az^3+Bz^2+Cz}{Dz^3+Ez^2+Fz+G}$ for $A \neq 0,B,C,D,E,F,G \in \mathbb{C}$ then $B=C=0$. 
		\item  For every rational map $\tilde{R}$, the Newton maps of $\tilde{R}$ and $  \tilde{R}(T)$ are conjugate whenever $T(z)=az+b$ is an affine map where $a,b   \in \mathbb{C},   \neq 0$. In fact, it is not difficult to see that $T(N_{  \tilde{R}(T)})=N_{\tilde{R}}(T)$~\cite{2}.
		\item If $a_1 \neq a_2$ such that $\tilde{R} ( a_1)=0=\tilde{R} (a_2)$  then by choosing $T(z)=( a_2 - a_1)z+a_1$, we see that $\tilde{R}(T)(0)=0=\tilde{R}(T)(1)$. Similarly, if $\tilde{R} ( a_1)=\infty=\tilde{R} ( a_2)$ then $\tilde{R}(T)(0)=\infty=\tilde{R}(T)(1)$ for the same $T$. 
		
		\item For $b_1 \neq b_2$, if $\tilde{R}(b_1)=0$ and $\tilde{R}(b_2)=\infty$ then taking $T(z)=(b_2 -b_1)z+b_1$, it is found that $\tilde{R}(T)(0)=0$ and  $\tilde{R}(T)(1)=\infty$.
	\end{enumerate}
\label{obs}
	\end{obs}
We use Observation~\ref{obs} (3-4) to simplify the form of the rational function $R$ and that is useful in deriving the conditions on it under which $N_R$ is conjugate to a polynomial. Recall that $m$ and $n$ denote the number of distinct roots and distinct poles of $R$ whereas $d$ and $e$ denote the sum of the multiplicities of all roots and of all poles of $R$ respectively. This is as  mentioned in the Proposition~\ref{NFP}. As mentioned in Theorem~\ref{2quadratic-Newtonmaps}, we assume $R=\frac{P}{Q}$ where both $P,Q$ are monic polynomials. 

\underline{\textbf{Case I  ($N_R(\infty) \neq \infty$):} }
If $N_R(\infty)\neq \infty$ then $m+n=4$ and $d=e+1$.
If either $m=0, n=4 $ or $m=4, n=0$  then  $d=e+1$ can not be true. We now look at all other possibilities. In this case, the exceptional point of $N_R$, whenever exists is a simple root of $R$. 
\begin{enumerate}
	\item \textbf{Subcase A ($m=1, n=3$)  }: The general form of the rational map is $\frac{(z-\alpha)^{d}}{(z-\beta_1)^{e_1}(z-\beta_2)^{e_2}(z-\beta_3)^{e_3}}$, where $d=e_1+e_2+e_3+1$ and $\alpha, \beta_1, \beta_2, \beta_3 \in \mathbb{C}$. By Observation~\ref{obs}(2) and (4), we assume without loss of any generality that $\alpha=0$ and $\beta_1 =1$ and the rational map is $R(z)=\frac{z^{d}}{(z-1)^{e_1}(z-\gamma_1)^{e_2}(z-\gamma_2)^{e_3}}$. 
Since $d>1$, $R$ has no simple finite root and consequently $N_R$ is never conjugate to a polynomial.
	\item \textbf{Subcase B ($m=2, n=2$)  }:
	The rational map is of the form $\frac{(z-\alpha_1)^{d_1}(z-\alpha_2)^{d_2}}{(z-\beta_1)^{e_1}(z-\beta_2)^{e_2}}$, where $ d_1+d_2=e_1+e_2+1$ and $\alpha_1, \alpha_2, \beta_1, \beta_2 \in \mathbb{C}$.
As mentioned in the begining, either $d_1 =1$ or $d_2 =1$. Assume that $d_1=1.$
In view of Observation~\ref{obs} (2) and (3), we assume $\alpha_1=0$ and $\alpha_2 =1$. The rational map becomes $ R(z)= \frac{z(z-1)^{e_1+e_2}}{(z-\gamma_1)^{e_1}(z-\gamma_2)^{e_2}}$ and its Newton map is 
	$$N_R (z)=\dfrac{(e_1+e_2-e_1\gamma_1-e_2\gamma_2)z^3+\{(e_1+e_2)\gamma
		_1\gamma_2-e_1\gamma_2-e_2\gamma_1\}z^2}{z^3-\{(1+e_1)\gamma_1+(1+e_2)\gamma_2+1-e_1-e_2\}z^2+K(z)}$$
		where $K(z)=\{(1-e_2)\gamma_1+(1-e_1)\gamma_2+(1+e_1+e_2)\gamma_1\gamma_2\}z-\gamma_1\gamma_2$.
		The point $0$ is exceptional for $N_R$ if and only if $(e_1+e_2)\gamma_1 \gamma_2 -e_1 \gamma_2 -e_2 \gamma_1=0$ by Observation~\ref{obs}(1). This is precisely when $N_R$ is conjugate to a polynomial. 
    \item \textbf{Subcase C ($m=3, n=1$)  }:
   In this situation the rational map is of the form $ \tilde{R}(z)=\frac{(z-\alpha_1)^{d_1}(z-\alpha_2)^{d_2}(z-\alpha_3)^{d_3}}{(z-\beta)^{e}} $ where $e=d_1+d_2+d_3-1$ and $\alpha_1, \alpha_2, \alpha_3, \beta \in \mathbb{C}$.
   As in the Subcase(B) above, we assume $d_1=1, \alpha_1 =0$ and $\alpha_2 =1$ without loss of generality.
    The rational map now becomes $ R(z)=\frac{z(z-1)^{d_2}(z-\alpha)^{d_3}}{(z-\gamma)^{d_2+d_3}} $. Its  Newton map is   $$ N_R(z)=\dfrac{\{d_2+d_3\alpha-(d_2+d_3)\gamma\}z^3+\{d_2\alpha\gamma+d_3\gamma-(d_2+d_3)\alpha\}z^2}{z^3-\{(1+d_2+d_3)\gamma-(d_3-1)\alpha-d_2+1\}z^2+K(z)} $$
    where $ K(z)=\{(1-d_2-d_3)\alpha+(1+d_3)\gamma+(1+d_2)\alpha\gamma\}z-\alpha\gamma$.   
    In order that $0$ is exceptional and hence $N_R$ is conjugate to a polynomial we must have $d_2\alpha \gamma+d_3\gamma-(d_2+d_3)\alpha=0$ and vice-versa.
\end{enumerate}
\underline{\textbf{Case II ($N_R(\infty)= \infty$)}:} If $N_R(\infty)= \infty$ then 
$m+n=3$, and $d\neq e+1$. In this case the exceptional point can be finite or $\infty$ depending on $R$.
\begin{enumerate}
    \item \textbf{Subcase A ($m=0, n=3$)  }:  The gneral form of the rational map is $ \frac{1}{(z-\beta_1)^{e_1}(z-\beta_2)^{e_2}(z-\beta_3)^{e_3}} $ for some $\beta_1, \beta_2, \beta_3 \in \mathbb{C}$. Here $d=0$ and $e=e_1+e_2+e_3$. By Observation~\ref{obs}(2) and (3), we assume $\beta_1=0$ and $\beta_2 =1$ and the rational map becomes $ R(z)=\frac{1}{z^{e_1}(z-1)^{e_2}(z-\gamma)^{e_3}} $.
    The Newton map is
    $$N_R(z)=\dfrac{(e+1)z^3-\{(e_1+e_2+1)\gamma+e_1+e_3+1\}z^2+(e_1+1)\gamma z}{ez^2-\{(e_1+e_2)\gamma +e_1+e_3\}z+e_1\gamma}.$$
    	Here $\infty$ is the only attracting fixed point with multiplier $\frac{e}{e+1}$ and can never be superattracting. Hence $N_R$ is never conjugate to a polynomial in this case.
   
    \item  \textbf{Subcase B ($m=1, n=2$)  }: 
   In this case, the rational map is of the form  $ \frac{(z-\alpha)^{d}}{(z-\beta_1)^{e_1}(z-\beta_2)^{e_2}}$ for some $\alpha, \beta_1, \beta_2 \in \mathbb{C}$. Using Observation~\ref{obs}(2) and (4), we assume without loss of generality that $\alpha=0$ and $\beta_1 =1$. Then the rational map becomes $ R(z)=\frac{z^{d}}{(z-1)^{e_1}(z-\gamma)^{e_2}} $.
     Here, the point  $\infty$ or $0$ is a superattracting fixed point when $d=e_1 +e_2 >1$ or $d=1$ respectively.
    \begin{enumerate}
    	\item [(i)] If $ \infty $ is superattracting then the Newton map turns out to be
    	$$N_R(z)=\dfrac{z^3+\{(e_2-1)\gamma +e_1-1\}z^2-(e_1+e_2-1)\gamma z}{(e_1+e_2\gamma)z-(e_1+e_2)\gamma}.$$
   The point $\infty$ is exceptional if and only if $N_R$ itself a polynomial, i.e.,  if and only if $\gamma =-\frac{e_1}{e_2}$.
  
    \item[(ii)] If $0$ is superattracting then the Newton map is 
   $$N_R(z)=\dfrac{(e_1+e_2)z^3-(e_1 \gamma+e_2)z^2}{(e_1+e_2-1)z^2-\{(e_1-1)\gamma +e_2-1\}z-\gamma}.$$
   As in Subcases B of Case(1),
    	$N_R$ is conjugate to a polynomial if and only if the coefficient of $z^2$ in the numerator is $0$, i.e.,  $\gamma =-\frac{e_2}{e_1}$.
   
    \end{enumerate}
    
    \item \textbf{Subcase C ($m=2, n=1$)  }: The general form of the rational map is $ \frac{(z-\alpha_1)^{d_1}(z-\alpha_2)^{d_2}}{(z-\beta)^{e}} $ for some $\alpha_1, \alpha_2, \beta \in \mathbb{C}$.
    Here the exceptional point can be $\infty$ or finite.  
    \begin{enumerate}
    	\item[(i)] The point $\infty$ is superattracting only when $e=d_1+d_2$. By Observation~\ref{obs}(2) and (3), we assume that $\alpha_1=0$ and $\alpha_2 =1$ and  we get the  rational map as $ R(z)=\frac{z^{d_1}(z-1)^{d_2}}{(z-\gamma)^{d_1 +d_2}} $.     	The Newton map is $$N_R(z)=\dfrac{z^3+\{(d_1+d_2-1)\gamma-d_2-1\}z^2-(d_1-1)\gamma z}{\{(d_1+d_2)\gamma-d_2\}z-d_1\gamma}.$$
    	The point $\infty$ is exceptional 
    		if $N_R$ itself a polynomial, i.e.,  if and only if $\gamma =\frac{d_2}{d_1+d_2}$.
      \item[(ii)] The Newton map has a finite superattracting fixed point when $d_1=1$ or $d_2 =1$. Without loss of generality, assume $d_1=1$. Further, by Observation~\ref{obs}(2) and (3), we assume $\alpha_1 =0$ and $\alpha_2=1$. Then  the   rational map becomes $ R(z)=\frac{z(z-1)^{d_2}}{(z-\gamma)^{e}} $. Its Newton map is
      $$N_R(z)=\dfrac{(d_2-e)z^3-(d_2\gamma-e)z^2}{(d_2-e+1)z^2-\{(1+d_2)\gamma-e+1\}z+\gamma}.$$
      The point $0$ is exceptional 
      	 if and only if $\gamma =\frac{e}{d_2}$ and this is precisely when  $N_R$ is conjugate to a polynomial.
    \end{enumerate}
    
    \item \textbf{Subcase D ($m=3, n=0$)}: The rational map is of the form $ (z-\alpha_1)^{d_1}(z-\alpha_2)^{d_2}(z-\alpha_3)^{d_3} $. In this case $\infty$ is a repelling fixed point of the Newton map. 
  It can have a superattracting (finite) fixed point if and only if the rational function has a simple root. Without loss of generality, assume $d_1=1$.
  In the light of Observation~\ref{obs}(2) and (3), we assume without loss of generality that $\alpha_1=0$ and $\alpha_2 =1$. The rational function becomes  $ R(z)=z(z-1)^{d_2}(z-\alpha)^{d_3} $. Its  Newton map is
    $$N_R(z)=\dfrac{(d_2+d_3)z^3-(d_2\alpha+d_3)z^2}{(d_2+d_3+1)z^2-\{(d_2+1)\alpha+d_3+1\}z+\alpha}$$
 As observed many times previously (for example in Case(1), Subcase B), the function
	$N_R$ is conjugate to a polynomial if and only if $\alpha =-\frac{d_3}{d_2}$.
	
	\par Table~\ref{table} describes all rational functions whose Newton maps are cubic. 
	
\begin{table}[h!]
\centering
\begin{tabular}{p{1.0cm} p{0.2cm} p{0.2cm} p{3.5cm} p{3.0cm} p{2.0cm}}
			\hline
			\multicolumn{5}{c}{Case-I: $N_R(\infty)\neq \infty$, $m+n=4$, $d=e+1$} \\
			\hline{\bf \vspace{0.01cm}Sub cases}&{\bf \vspace{0.1cm}$m$}&{\bf \vspace{0.1cm}$n$}&{\bf \vspace{0.1cm}$R(z)$}&{\bf \vspace{0.01cm} When  $N_R \sim$ polynomial?}&{\bf \vspace{0.01cm} Exceptional point} \\
			\hline \vspace{0.1cm}A & \vspace{0.1cm}$1$& \vspace{0.1cm}$ 3 $ & \vspace{0.1cm} $\frac{z^{e_1+e_2+e_3+1}}{(z-1)^{e_1}(z-\gamma_1)^{e_2}(z-\gamma_2)^{e_3}}$  &\vspace{0.1cm} Never&\vspace{0.1cm} NA\\
			\vspace{0.1cm}B & \vspace{0.1cm}$2$ & \vspace{0.1cm}$2$ & \vspace{0.1cm} $\frac{z(z-1)^{e_1+e_2}}{(z-\gamma_1)^{e_1}(z-\gamma_2)^{e_2}}$ &\vspace{0.1cm} $(e_1+e_2)\gamma_1 \gamma_2 -e_1 \gamma_2 -e_2 \gamma_1=0$ &\vspace{0.1cm}$0$\\
			\vspace{0.1cm}C & \vspace{0.1cm}$ 3 $ & \vspace{0.1cm}$ 1 $ & \vspace{0.1cm}$ \frac{z(z-1)^{d_2}(z-\alpha)^{d_3}}{(z-\gamma)^{d_2+d_3}} $ & \vspace{0.1cm}$d_2\alpha \gamma+d_3\gamma-(d_2+d_3)\alpha=0$ &\vspace{0.1cm}$0$\\
			\hline
			\multicolumn{5}{c}{Case-II: $N_R(\infty)=\infty$, $m+n=3$, $d\neq e+1$} \\
			\hline
			\vspace{0.1cm}A & \vspace{0.1cm}$0$ & \vspace{0.1cm}$3$& \vspace{0.1cm} $ \frac{1}{z^{e_1}(z-1)^{e_2}(z-\gamma)^{e_3}} $ &\vspace{0.1cm} Never&\vspace{0.1cm} NA\\
			\vspace{0.1cm}B (i) & \vspace{0.1cm}$1$ & \vspace{0.1cm}$2$& \vspace{0.1cm} $ \frac{z^{e_1+e_2}}{(z-1)^{e_1}(z-\gamma)^{e_2}}$ &\vspace{0.1cm} $\gamma =-\frac{e_1}{e_2}$&\vspace{0.1cm} $\infty$\\
			\vspace{0.1cm}B (ii) & \vspace{0.1cm}$1$ & \vspace{0.1cm}$2$ & \vspace{0.1cm} $ \frac{z}{(z-1)^{e_1}(z-\gamma)^{e_2}} $& \vspace{0.1cm}$\gamma =-\frac{e_2}{e_1}$&\vspace{0.1cm} $0$\\
			\vspace{0.1cm} C (i) & \vspace{0.1cm}$2$ & \vspace{0.1cm}$1$ & \vspace{0.1cm} $ \frac{z^{d_1}(z-1)^{d_2}}{(z-\gamma)^{d_1+d_2}} $& \vspace{0.1cm}$\gamma =\frac{d_2}{d_1+d_2}$ &\vspace{0.1cm}$\infty$\\
			\vspace{0.1cm}C (ii) & \vspace{0.1cm}$2$ & \vspace{0.1cm}$1$ & \vspace{0.1cm} $ \frac{z(z-1)^{d_2}}{(z-\gamma)^{e}} $& \vspace{0.1cm} $\gamma =\frac{e}{d_2}$ &\vspace{0.1cm}$0$\\
			\vspace{0.1cm} D & \vspace{0.1cm}$3$ &\vspace{0.1cm} $0$ & \vspace{0.1cm} $ z(z-1)^{d_2}(z-\alpha)^{d_3} $& \vspace{0.1cm} $\alpha =-\frac{d_3}{d_2}$&\vspace{0.1cm} $0$ \\
			\hline
\end{tabular}
\caption{\label{table} All $R$ for which $N_R$ is cubic.}
	\end{table}
     
\end{enumerate}

 A quadratic Newton map $N_R$ has three distinct fixed points. It is observed that at least one of them is attracting and one is repelling. If the third fixed point is attracting then its Julia set is a Jordan curve. Otherwise the Julia set is totally disconnected. This is a restatement of Theorem~\ref{main}. If a cubic Newton map is conjugate to a polynomial then, in addition to the superattracting fixed point it has other three fixed points. At least one of these three is attracting and another is repelling. If the third fixed point is attracting then its Julia set is a self-intersecting closed curve. The Julia set is a Jordan curve whenever the third fixed is repelling. Therefore, the dynamics of all quadratic Newton maps and all cubic Newton maps conjugate to polynomials is completely determined by the nature of their fixed points. This is unlikely to be true in general. The simplest case that can be looked into is that of all cubic Newton maps which are not conjugate to any polynomial. All the rational functions leading to such a cubic Newton map can be found in Table~\ref{table}.

\section*{Acknowledgement}
The first author acknowledges SERB, Govt. of India for financial support through a MATRICS project. The second author is supported by the University Grants Commission, Govt. of India.

\end{document}